\documentclass{amsart}

\usepackage{hyperref}
\hypersetup{colorlinks=true,
	linkcolor=black, 
	urlcolor=blue} 
\usepackage{amsmath,amssymb,amsfonts,amsthm}
\usepackage{amsthm}
\usepackage{enumerate}
\numberwithin{equation}{section}
\usepackage{marginnote}
\usepackage[a4paper]{geometry}

\reversemarginpar

\usepackage[english]{babel}
\usepackage[babel, english=american]{csquotes}

\newtheoremstyle{Teorema}
{3pt}
{3pt}
{\slshape}
{}
{\bfseries}
{:}
{\newline}
{}

\newtheorem{theorem}{Theorem}[section]
\newtheorem{definition}{Definition}

\newtheorem{corollary}[theorem]{Corollary}

\newtheorem{lemma}[theorem]{Lemma}

\theoremstyle{definition}

\DeclareMathOperator{\lcm}{lcm}
\DeclareMathOperator{\D}{d}
\DeclareMathOperator{\supp}{supp}
\DeclareMathOperator{\Co}{Co}
\DeclareMathOperator{\capac}{cap}

\newcommand{\N}{\mathbb{N}}
\newcommand{\Z}{\mathbb{Z}}
\newcommand{\R}{\mathbb{R}}
\newcommand{\C}{\mathbb{C}}
\newcommand{\Reg}{\mathbf{Reg}}
\newcommand{\diff}{\backslash}

\title{Logarithmic asymptotic of multi-level Hermite-Pad\'e polynomials}
\date{\today}

\begin{document}
\author{L.G. Gonz\'alez Ricardo}
\address{Department of Mathematics, Universidad Carlos III de Madrid, Avda. Universidad, 30
CP-28911, Leganés, Madrid, Spain.}
\email{luisggon@math.uc3m.es}

\author{G. L\'opez Lagomasino}
\address{Department of Mathematics, Universidad Carlos III de Madrid, Avda. Universidad, 30
CP-28911, Leganés, Madrid, Spain.}
\email{lago@math.uc3m.es}
\thanks{The second author was supported by research grant PGC2018-096504-B-C33 of Ministerio de Ciencia, Innovaci\'on y Universidades, Spain.}

\author{S. Medina Peralta}
\address{Department of Mathematics and Statistics, Florida International University,
DM 430 11200 SW 8th Street Miami, Florida 33199, USA.
}\email{smedinaperalta@gmail.com}

\maketitle
\begin{abstract}

We study the logarihtnmic asymptotic of multiple orthogonal polynomials arising in a mixed type Hermite-Pad\'e approximation problem associated with the rational perturbation of a Nikishin system of functions. The formulas obtained allow to give exact estimates of the rate of convergence of the corresponding Hermite-Pad\'e approximants.

	\textbf{Keywords:} Nikishin system, multiple orthogonal polynomials, Hermite-Pad\'e approximation, logarithmic asymptotic
\end{abstract}

\section{Introduction}

The birth of Hermite-Pad\'e approximation is linked with their application in number theory. In recent years, this scheme of approximation has been found to be useful in many other fields of mathematics. The type of Hermite-Pad\'e polynomials studied here first appear in the problem of finding discrete solutions to the Degasperis-Procesi equation, see \cite{LS} and \cite{3}. Motivated by that application the definition was extended to general Nikishin systems of functions in \cite{Lago_Sergio_Jacek} where their convergence was proved. Later, see \cite{Lago_Sergio_Ulises}, the logarithmic and ratio asymptotic behavior of the associated Hermite-Pad\'e polynomials was given. In \cite{GLM}, we further extended the definition  to systems of functions obtained through the rational perturbation of Nikishin systems and proved their convergence. See \cite{Lago_Sergio1} for an analogous problem related with type I Hermite-Pad\'e approximation. Here, we  provide their logarithmic asymptotic behavior and use it to give exact estimates of the rate of convergence.

\subsection{Nikishin systems}

Nikishin systems of functions were introduced in \cite{nikishin}. Such systems of functions have proved to be appropriate in the attempt of extending the general theory of orthogonal polynomials on the real line and the theory of Pad\'e approximation to multiple orthogonal polynomials and Hermite-Pad\'e approximation.

In the sequel, we will only consider Borel measures $s$ with constant sign, finite moments $c_n = \int x^n \D s (x), |c_n| < \infty$, $n\in\Z_+$, whose support consists of infinitely many points, and $\supp s\subset \R$. We will denote by $\Delta$ the smallest interval which contains $\supp s$, i.e. its convex hull. The class of these measures will be denoted by $\mathcal{M}(\Delta)$. Let
$$\widehat{s}(z) = \int \frac{\D s(x)}{z-x}$$
denote the Cauchy transform of the measure $s$. Obviously, $\widehat{s}(z)$ is holomorphic in $\C\diff\Delta$ and we can associate to $\widehat{s}$ its formal Taylor expansion at infinity
$$\widehat{s}(z)\sim \sum_{j=0}^\infty\frac{c_j}{z^{j+1}},\qquad c_j = \int x^j\D s(x).$$
In the present work, $s$ is finite and has compact support; therefore, its moments are all finite and the expansion above is convergent in a neighborhood of infinity.

Let $\Delta_\alpha$, $\Delta_\beta$ be two compact intervals contained in $\mathbb{R}$ such that $\Delta_\alpha\cap\Delta_\beta=\emptyset$. Consider the measures $\sigma_\alpha\in\mathcal{M}(\Delta_\alpha), \sigma_\beta\in\mathcal{M}(\Delta_\beta)$. Define
$$\D\langle \sigma_\alpha,\sigma_\beta\rangle(x): = \widehat{\sigma}_\beta(x)\D\sigma_\alpha(x).$$
This product of measures is neither commutative nor associative.

\begin{definition}
\label{Nikishin_sys}
Take a collection $\Delta_j$, $j=1,\ldots,m$ of bounded intervals such that
$$\Delta_j\cap\Delta_{j+1}=\emptyset,\; j=1,\ldots,m-1.$$
Let $(\sigma_1,\ldots,\sigma_m)$ be a system of measures such that $\Co(\supp \sigma_j)=\Delta_j$, $\sigma_j\in\mathcal{M}(\Delta_j)$, $j=1,\ldots,m$. We say that  $(s_{1,1},\ldots,s_{1,m})=\mathcal{N}(\sigma_1,\ldots,\sigma_m)$, where
$$s_{1,1}=\sigma_1,\quad s_{1,2}=\langle\sigma_1,\sigma_2 \rangle,\; \ldots, \quad s_{1,m}=\langle \sigma_1, \langle \sigma_2,\ldots,\sigma_m \rangle\rangle,$$
is the Nikishin system of measures generated by $(\sigma_1,\ldots,\sigma_m)$. The vector $(\widehat{s}_{1,1},\ldots,\widehat{s}_{1,m})$ is called a Nikishin system of functions.
\end{definition}

In \cite{Lago_Sergio_Jacek} and \cite{GLM} the generating measures of the Nikishin system are allowed to be supported on unbounded intervals. The results of this paper require that the supports be bounded so, for simplicity, we have included that restriction in the definition as was done in the original version \cite{nikishin}.  In the sequel, for $1\leq j\leq k\leq m$ we write
$$s_{j,k} := \langle \sigma_j, \sigma_{j+1},\ldots, \sigma_k \rangle,\qquad s_{k,j} := \langle \sigma_k, \sigma_{k-1},\ldots, \sigma_j \rangle.$$

\subsection{Multi-level Hermite-Pad\'e polynomials}

Let us  define the approximation objects.
\begin{definition}\label{defML}
Consider the Nikishin system $\mathcal{N}(\sigma_1,\ldots,\sigma_m)$. Let $\displaystyle r_j = {v_j}/{t_j}$, $k=1,\ldots,m$, be rational fractions with real coefficients, $\deg v_k<\deg t_k=d_k$, $(v_k,t_k)=1$ (coprime) for all $k=1,\ldots,m$. For each $n\in\N$, there exist polynomials $a_{n,0}, a_{n,1},\ldots,a_{n,m}$, with $\deg a_{n,j}\leq n-1$, $j=0,1,\ldots,m-1, \deg a_{n,m}\leq n$, not all identically equal to zero, called multi-level (ML) Hermite-Pad\'e polynomials that verify
\begin{align}
    \mathcal{A}_{n,0} :=& \left[ a_{n,0} + \sum_{k=1}^m (-1)^k a_{n,k}(\widehat{s}_{1,k}+r_k) \right] \in \mathcal{O}\left(\frac{1}{z^{n+1}}\right), \label{Problem1}\\
    \mathcal{A}_{n,j} :=& \left[(-1)^j a_{n,j} + \sum_{k=j+1}^m (-1)^k a_{n,k}\widehat{s}_{j+1,k}\right] \in \mathcal{O}\left(\frac{1}{z}\right), \quad j=1,\ldots,m-1.\label{Problem2}
\end{align}
Here and in the sequel $\mathcal{O}(\cdot)$ is as $z \to \infty$ along paths non tangential to the support of the measures involved. For completeness we denote $\mathcal{A}_{n,m} := (-1)^m a_{n,m}$.
\end{definition}

When $r_k \equiv 0, k=1,\ldots,m,$ this construction was introduced in \cite{Lago_Sergio_Jacek}.
Notice that the interpolation conditions involve the Nikishin systems $\mathcal{N}(\sigma_1,\ldots,\sigma_m)$, $\mathcal{N}(\sigma_2,\ldots,\sigma_m)$, \ldots, $\mathcal{N}(\sigma_m) = (s_{m,m})$. It is easy to justify that for each $n \in \mathbb{N}$ the ML Hermite-Pad\'e polynomials exist; however, they are not uniquely determined.

Without loss of generality, we can assume that the polynomials $t_k$ are monic.
Let $T = \lcm (t_1,\ldots,t_m), \deg T = D$,   where $\lcm$ stands for least common multiple. Set
\[ f := \widehat{s}_{m,1} - \sum_{k=1}^{m-1} (-1)^{k}\widehat{s}_{m,k+1}r_k - (-1)^m r_m.
\]

In \cite[Theorem 1.1]{GLM} it was proved that if the zeros $T$ lie in the complement of $\Delta_1 \cup \Delta_m$ and $f$ has exactly $D$ poles in $\C \setminus \Delta_m$,  then
\begin{equation} \label{convergencia}
\lim_{n\to \infty}\frac{a_{n,j}}{a_{n,m}} = \widehat{s}_{m,j+1}, \quad j=1,\ldots,m-1, \qquad   \lim_{n\to \infty} \frac{a_{n,0}}{a_{n,m}} = f(z)
\end{equation}
uniformly on compact subsets of $\C\diff(\Delta_m \cup \{z:T(z) = 0\})$. Under the present assumptions,  $f$  has $D$ poles in $\C \setminus \Delta_m$ if and only if for each $\zeta, T(\zeta) = 0$,
\[ \lim_{z \to \zeta} (z-\zeta)^{\tau}f(z) =   - \sum_{k=1}^{m-1} (-1)^{k}\widehat{s}_{m,k+1}(\zeta) \lim_{z\to \zeta} (z-\zeta)^\tau r_k(z) - (-1)^m \lim_{z\to \zeta} (z-\zeta)^\tau r_m(z) \neq 0.
\]
where $\tau$ is the multiplicity of $\zeta$. This is true, for example,  if $(t_j,t_k) = 1, 1\leq j,k \leq m$.

In this paper, we  have assumed that the intervals $\Delta_j$ (in particular $\Delta_m$) are bounded, so convergence takes place in \eqref{convergencia} with geometric rate, see \cite[Corollary 3.4]{GLM}. We aim  to provide the exact order of convergence  (see Theorem \ref{convergence_speed} below). For this purpose, we need to study the logarithmic asymptotic of the sequence of polynomials $\left(a_{n,m}\right)_{n \in \mathbb{N}}$ and the sequences of  forms $\left(\mathcal{A}_{n,j}\right)_{n \in \mathbb{N}}, j=0,\ldots,m-1.$
This is done using methods of potential theory.

\subsection{Statement of the main results}

In the sequel, we assume that $\supp \sigma_k$ is a regular compact set for $k=1,\ldots,m$; that is,  Green's function of the region $\C \setminus \supp \sigma_k$ with singularity at $\infty$ can be extended continuously to $\supp \sigma_k$. Let $\mathcal{M}_1(\supp \sigma_k)$ be the subclass of probability measures in $\mathcal{M}(\supp \sigma_k)$. Define
$$\mathcal{M}_1 = \mathcal{M}_1(\supp \sigma_1)\times \cdots \times \mathcal{M}_1(\supp \sigma_m).$$
Let
\[ V^{\mu}(z) := \int \log\frac{1}{|z-x|} \D \mu (x)
\]
denote the logarithmic potential of the measure $\mu$.

It is well known (see, for example, \cite[Section 4]{bello_lago}), that there exists a unique vector measure $\vec{\lambda}= (\lambda_1,\ldots,\lambda_m) \in \mathcal{M}_1$ and a unique vector constant $\omega^{\vec{\lambda}} = (\omega_1^{\vec{\lambda}},\ldots,\omega_m^{\vec{\lambda}})$ such that
\begin{equation} -\frac{1}{2} V_{j-1}^{\vec{\lambda}}(x) +  V_{j}^{\vec{\lambda}}(x)-\frac{1}{2} V_{j+1}^{\vec{\lambda}}(x)= \omega_j^{\vec{\lambda}},\qquad x\in\supp \lambda_j, \qquad j=1,\ldots,m.  \label{vecequil}
\end{equation}
(By convention $V_{0}^{\vec{\lambda}} \equiv V_{m+1}^{\vec{\lambda}} \equiv 0$.) The vector measure $\vec{\lambda}$ is called equilibrium measure for the system of compact sets $\supp \sigma_k , k=1,\ldots,m$ with interaction matrix $\mathcal{C}_\mathcal{N} = \left(c_{j,k}\right), 1\leq,j,k\leq m$, where
$c_{j,j} = 1, j=1,\ldots,m, c_{j-1,j}= -1/2, j=2,\ldots,m, c_{j,j+1} = -1/2, j=1,\ldots,m-1$, and the rest of entries equal zero. Notice that the left hand of \eqref{vecequil} is the product of the $j$-th row of $\mathcal{C}_\mathcal{N}$ times the vector potential $(V_{1}^{\vec{\lambda}},\ldots, V_{m}^{\vec{\lambda}})$.

The vector equilibrium measure allows to describe the normalized distribution of the zeros of the polynomials $a_{n,m}$ and roots of the forms $\mathcal{A}_{n,j},j=1,\ldots,m-1$. From \cite[Theorem 1.1, Corollary 3.1]{GLM} it follows that when all the zeros of $T$ lie in the complement of $\Delta_1 \cup \Delta_m$ and $f$ has exactly $D$ poles in $\C \setminus \Delta_m$ then for all sufficiently large $n > N$:
\begin{itemize}
\item $\deg a_{n,m} = n$ with exactly $n-D$ simple zeros on $\Delta_m$ and the remaining $D$ zeros of $a_{n,m}$ converge to the poles of $f$ in $\C \setminus \Delta_m$ according to their multiplicity.
\item $\mathcal{A}_{n,j}, j=1,\ldots,m-1,$ has exactly $n-D$ zeros in $\C \setminus \Delta_{j+1}$ they are simple and lie in $\Delta_j$.
\end{itemize}
In the rest of the paper, we assume that the zeros of $T$ lie in the complement of $\Delta_1 \cup \Delta_m$, that $f$ has exactly $D$ poles in $\C \setminus \Delta_m$, and $n > N$.

Given $n \in \mathbb{N}$, if any solution of \eqref{Problem1}-\eqref{Problem2} has $\deg a_{n,m} = n$ then it is easy to verify that $(a_{n,0},\ldots,a_{n,m})$ is uniquely determined except for a constant factor. Without further notice, we normalize $(a_{n,0},\ldots,a_{n,m}), n > N,$ so that $a_{n,m}$ is monic.

Let $Q_{n,j}, j=1,\ldots,m$ be the monic polynomial of degree $n-D$ whose zeros are the roots of $\mathcal{A}_{n,j}$ on $\Delta_j$. (Recall that $\mathcal{A}_{n,m} = (-1)^m a_{n,m}$.)
\begin{equation}
     \mathcal{H}_{n,j} = \frac{Q_{n,j+1}T\mathcal{A}_{n,j}}{Q_{n,j}}, \quad j=0,1, \qquad \; \mathcal{H}_{n,j} = \frac{Q_{n,j+1}\mathcal{A}_{n,j}}{Q_{n,j}},\quad j=2,\ldots,m.\nonumber
\end{equation}
By convention $Q_{n,0} \equiv Q_{n,m+1} \equiv 1$ and $\Delta_{m+1} = \emptyset$. Notice that $\mathcal{A}_{n,j}/Q_{n,j} \in \mathcal{H}(\C \setminus \Delta_{j+1}), j=0,\ldots,m$.

Given a polynomial $Q$ the associated normalized zero counting measure is denoted
\begin{equation*}
    \mu_Q := \frac{1}{\deg(Q)}\sum_{Q(x)=0}\delta_x,
\end{equation*}
where $\delta_x$ is the Dirac measure with mass $1$ at $x$,

A measure $\sigma\in\mathcal{M}(\Delta)$ is called regular if
\begin{equation*}
    \lim_{n\to \infty} \gamma_n^{1/n} = \frac{1}{\capac(\supp\sigma)}
\end{equation*}
where $\capac(\supp\sigma)$ denotes the logarithmic capacity of the support of $\sigma$ and $\gamma_n$ is the leading coefficient of the $n$-th orthonormal polynomial with respect to $\sigma$. This condition will be denoted $\sigma\in\Reg$. Many equivalent forms of defining regular measures may be seen in \cite[Chapter 3]{Stahl_Totik}.

\begin{theorem} \label{teo:1} Assume that all the zeros of $T$ lie in the complement of $\Delta_1 \cup \Delta_m$, and $f$ has exactly $D$ poles in $\C \setminus \Delta_m$. Suppose that $\sigma_j\in \Reg$ and $\supp\sigma_j, j=1,\ldots,m$. Then,
\begin{equation}
    *\lim_{n\to \infty} \mu_{Q_{n,j}}=\lambda_j,\;\; j=1,\ldots,m \label{weak_Qnj}
\end{equation}
where $\vec{\lambda} = (\lambda_1,\ldots,\lambda_m)\in\mathcal{M}_1$ is the vector equilibrium measure determined by the matrix $\mathcal{C}_\mathcal{N}$ on the system of compact sets $\supp\sigma_j$, $j=1,\ldots,m$. Moreover,
\begin{equation}
    \lim_{n\to \infty} \left| \int Q^2_{n,j}(x)\frac{\mathcal{H}_{n,j}(x)\D \sigma_j(x)}{Q_{n,j-1}(x)Q_{n,j+1}(x)}\right|^{1/2n} = \exp\left(-\sum_{k=j}^m \omega_k^{\vec{\lambda}}\right), \label{limit_nesimo}
\end{equation}
where $\omega^{\vec{\lambda}} = (\omega_1^{\vec{\lambda}},\ldots,\omega_m^{\vec{\lambda}})$ is the vector equilibrium constant.
\end{theorem}

From this result the logarithmic asymptotic behavior of the forms $\mathcal{A}_{n,j}$ can be derived.

\begin{theorem}
\label{logarithmic_asymptotic_Anj} Suppose that the assumptions of Theorem \ref{teo:1} are satisfied. Then,
\begin{equation}
\label{limit_nrooth_Anj}
    \lim_{n\to \infty} |\mathcal{A}_{n,j}(z)|^{1/n} = A_j(z),\qquad \mathcal{K}\subset\C\diff(\Delta_j\cup\Delta_{j+1}),\qquad j=1,\ldots,m-1
\end{equation}
where
\begin{equation*}
    A_j(z) = \exp\left(V^{\lambda_{j+1}}(z) - V^{\lambda_j}(z)-2\sum_{k=j+1}^m \omega_k^{\vec{\lambda}}\right), \qquad j=1,\ldots,m-1.
\end{equation*}
Moreover,
 \begin{equation*}
    \lim_{n\to \infty} |\mathcal{A}_{n,m}(z)|^{1/n} = \exp\left(-V^{\lambda_m}(z)\right),\qquad \mathcal{K}\subset\C\diff(\Delta_m\cup Z).
\end{equation*}
where $Z = \{z: T(z) =0\}$, and
\begin{equation*}
    \lim_{n\to \infty} |\mathcal{A}_{n,0}(z)|^{1/n} = \exp\left(V^{\lambda_1}(z)-2\sum_{k=1}^m \omega^{\vec{\lambda}}_k\right),\qquad \mathcal{K}\subset\C\diff(\Delta_1\cup Z).
\end{equation*}
$\vec{\lambda} = (\lambda_1,\ldots,\lambda_m)$ is the vector equilibrium measure and $(\omega^{\vec{\lambda}}_1,\ldots,\omega^{\vec{\lambda}}_m)$ is the vector equilibrium constant for the vector potential problem determines by the interaction matrix $\mathcal{C}_{\mathcal{N}}$ acting on the system of compact sets $\supp \sigma_j$, $j=1,\ldots,m$.
\end{theorem}

Theorem \ref{logarithmic_asymptotic_Anj} allows us to provide in Theorem \ref{convergence_speed} the asymptotic estimates we look for.

\section{Auxiliary results}

\subsection{Some integral representations}

We begin by obtaning some integral representations which will be needed.

\begin{lemma}
\label{zeros_A_nj}
Assume that all the zeros of $T$ lie in the complement of $\Delta_1 \cup \Delta_m$, $f$ has exactly $D$ poles in $\C \setminus \Delta_m$, and $n > N \geq D$. Then, for each $j=1,\ldots,m-1$,
\begin{equation}
\label{A_njQ_nj_1}
    \frac{\mathcal{A}_{n,j}}{Q_{n,j}}(z) = \int_{\Delta_{j+1}}\frac{\mathcal{A}_{n,j+1}(x)}{z-x}\frac{\D\sigma_{j+1}(x)}{Q_{n,j}(x)}
\end{equation}
and
\begin{equation}
\label{A_njQ_nj_2}
    T(z) \mathcal{A}_{n,0}(z) =  \int_{\Delta_1}\frac{\mathcal{A}_{n,1}(x)T(x)}{z-x}\D\sigma_1(x).
\end{equation}
Moreover, for $j=1,\ldots,m-1$
\begin{equation}
\label{OrthA_njQ_nj_1}
    \int_{\Delta_{j+1}}x^\nu\mathcal{A}_{n,j+1}(x)\frac{\D\sigma_{j+1}(x)}{Q_{n,j}(x)} = 0,\qquad \nu=0,1,\ldots,n-D-1.
\end{equation}
and
\begin{equation}
\label{OrthA_njQ_nj_2}
     \int_{\Delta_1} x^\nu \mathcal{A}_{n,1}(x)T(x)\D\sigma_1(x)=0,\qquad \nu=0,1,\ldots,n-D-1.
\end{equation}
\end{lemma}
\begin{proof} Notice that $T\mathcal{A}_{n,0} = \mathcal{O}(1/z^{n-D+1}) \in \mathcal{H}(\C \setminus \Delta_1)$. Let $\Gamma$ be a positively oriented closed Jordan curve which surrounds $\Delta_1$ so that $\Delta_2$ and $z$ remain in the unbounded connected component of the complement of $\Gamma$. We have
\[T(z)\mathcal{A}_{n,0}(z)= \frac{1}{2\pi i}\int_{\Gamma} \frac{(T\mathcal{A}_{n,0})(\zeta)}{z - \zeta}  \D \zeta = \frac{1}{2\pi i}\int_{\Gamma} \frac{(Ta_{n,0} + \sum_{k=1}^m (-1)^k a_{n,k}T r_k)(\zeta)}{z - \zeta} \D \zeta\]
\[ + \frac{1}{2\pi i}\int_{\Gamma} \frac{(T\sum_{k=1}^m (-1)^k a_{n,k}\widehat{s}_{1,k})(\zeta)}{z - \zeta} \D \zeta =
\]
\[ \int \frac{1}{2\pi i}\int_{\Gamma} \frac{(-T a_{n,1} + T\sum_{k=2}^m (-1)^k a_{n,k}\widehat{s}_{2,k})(\zeta)}{(z - \zeta)(\zeta -x)} \D \zeta \D \sigma_1(x) = \int \frac{(T\mathcal{A}_{n,1})(x)}{z-x} d \sigma_1(x).
\]
Indeed the first equality comes from Cauchy's integral formula for the complement of $\Gamma$. The second equality is trivial. Here, the first integral is zero since $(Ta_{n,0} + \sum_{k=1}^m (-1)^k a_{n,k}T r_k)(\zeta)/(z-\zeta)$ is analytic with respect to $\zeta$ inside $\Gamma$. Substituting in the second integral $\widehat{s}_{1,k}$ with its integral representation and using Fubini's theorem you get the third equality. The last equality comes from the use of Cauchy's integral formula inside $\Gamma$. Thus we obtain \eqref{A_njQ_nj_2}.

Similarly, since  $z^{\nu}T\mathcal{A}_{n,0} = \mathcal{O}(1/z^{2}) \in \mathcal{H}(\C \setminus \Delta_1), \nu=0,\ldots,n-D-1$, we obtain
\[ 0 = \frac{1}{2\pi i}\int_{\Gamma}  {\zeta^{\nu}(T\mathcal{A}_{n,0})(\zeta)}  \D \zeta =
\]
\[ \int \frac{1}{2\pi i}\int_{\Gamma} \frac{\zeta^{\nu}(-T a_{n,1} + T\sum_{k=2}^m (-1)^k a_{n,k}\widehat{s}_{2,k})(\zeta)}{ (\zeta -x)} \D \zeta \D \sigma_1(x) = \int  {x^{\nu}(T\mathcal{A}_{n,1})(x)}  d \sigma_1(x).
\]
which is \eqref{OrthA_njQ_nj_2}.

In order to derive \eqref{A_njQ_nj_1} and \eqref{OrthA_njQ_nj_1} one proceeds analogously. It is sufficient to use that
$z^{\nu}\mathcal{A}_{n,j}/Q_{n,j} = \mathcal{O}(1/z^2) \in \mathcal{H}(\C \setminus \Delta_{j+1}), j=1,\ldots,m-1, \nu =0,\ldots,n-D-1$ and take $\Gamma$ a positively oriented closed Jordan curve which surrounds $\Delta_{j+1}$ so that $\Delta_{j+2}\, (\Delta_{m+`1} = \emptyset)$ and $z$ remain in the unbounded connected component of the complement of $\Gamma$. The details are left to the reader.
\end{proof}

The previous lemma can be reformulated as follows.

\begin{lemma} Assume that all the zeros of $T$ lie in the complement of $\Delta_1 \cup \Delta_m$, $f$ has exactly $D$ poles in $\C \setminus \Delta_m$, and $n > N \geq D$. For each fixed  $j=0,\ldots,m-1$,
\begin{equation}
\label{Orth_H_nj}
    \int x^\nu Q_{n,j+1}(x)\frac{\mathcal{H}_{n,j+1}(x)\D\sigma_{j+1}(x)}{Q_{n,j}(x)Q_{n,j+2}(x)}=0, \qquad \nu=0,1,\ldots,n-D-1.
\end{equation}
Moreover, for $j=0,2,3,\ldots,m-1$
\begin{equation}
\label{H_mjIdentity_1}
    \mathcal{H}_{n,j}(z) = \int \frac{Q^2_{n,j+1}(x)}{z-x}\frac{\mathcal{H}_{n,j+1}(x)\D\sigma_{j+1}(x)}{Q_{n,j}(x)Q_{n,j+2}(x)},
\end{equation}
and
\begin{equation}
\label{H_mjIdentity_2}
    \mathcal{H}_{n,1}(z) = T(z)\int \frac{Q^2_{n,2}(x)}{z-x}\frac{\mathcal{H}_{n,2}(x)\D\sigma_{2}(x)}{Q_{n,1}(x)Q_{n,3}(x)}.
\end{equation}
Recall that by convention $Q_{n,0} \equiv Q_{n,m+1} \equiv 1$.
\end{lemma}
\begin{proof}
Formula (\ref{Orth_H_nj}) is  a restatement of  (\ref{OrthA_njQ_nj_1}) and (\ref{OrthA_njQ_nj_2}) using the notation of the functions $\mathcal{H}_{n,j}$.

Since $\deg Q_{n,j+1} =n-D$, from (\ref{Orth_H_nj}) we deduce that for $j=0,\ldots,m-1$
\begin{equation}
    \int \frac{Q_{n,j+1}(z)-Q_{n,j+1}(x)}{z-x}Q_{n,j+1}(x)\frac{\mathcal{H}_{n,j+1}(x)\D\sigma_{j+1}(x)}{Q_{n,j}(x)Q_{n,j+2}(x)}=0.\nonumber
\end{equation}
This last identity can be rewritten as
\begin{equation*}
    Q_{n,j+1}(z)\int \frac{Q_{n,j+1}(x)}{z-x}\frac{\mathcal{H}_{n,j+1}(x)\D\sigma_{j+1}(x)}{Q_{n,j}(x)Q_{n,j+2}(x)} = \int \frac{Q^2_{n,j+1}(x)}{z-x}\frac{\mathcal{H}_{n,j+1}(x)\D\sigma_{j+1}(x)}{Q_{n,j}(x)Q_{n,j+2}(x)}.
\end{equation*}
For $j=2,\ldots,m-1$
\begin{equation*}
    \int \frac{Q_{n,j+1}(x)}{z-x}\frac{\mathcal{H}_{n,j+1}(x)\D\sigma_{j+1}(x)}{Q_{n,j}(x)Q_{n,j+2}(x)} = \int\frac{\mathcal{A}_{n,j}(x)}{z-x}\frac{\D\sigma_{j+1}(x)}{Q_{n,j}(x)} = \frac{\mathcal{A}_{n,j}(z)}{Q_{n,j}(z)}
\end{equation*}
and (\ref{H_mjIdentity_1}) immediately follows for $j=2,\ldots,m-1$. In the case $j=1$, notice that
\begin{equation*}
\int \frac{Q_{n,2}^2(x)}{z-x}\frac{\mathcal{H}_{n,2}(x)\D\sigma_2(x)}{Q_{n,1}(x)Q_{n,3}(x)} = \frac{Q_{n,2}(z)}{Q_{n,1}(z)}\mathcal{A}_{n,1}(z) = \frac{ \mathcal{H}_{n,1}(z)}{T(z)} ,
\end{equation*}
which is equivalent to (\ref{H_mjIdentity_2}). For $j=0$ we proceed as for (\ref{H_mjIdentity_1}),$j=2,\ldots,m-1$.
\end{proof}

The previous lemma indicates that the polynomial $Q_{n,j}$, $j=1,\ldots,m$, is orthogonal with respect to the varying measure
\begin{equation*}
    \frac{\mathcal{H}_{n,j}(x)\D\sigma_{j}(x)}{Q_{n,j-1}(x)Q_{n,j+1}(x)}.
\end{equation*}
This varying measure has constant sign because  $Q_{n,j-1}$ and $Q_{n,j+1}$ have constant sign on $\Delta_{j}$ and $\mathcal{H}_{n,j}$ also has constant sign since $Q_{n,j}$ takes away the zeros of $\mathcal{A}_{n,j}$ on $\mathring{\Delta}_{j}$.

\subsection{Preliminaries from potential theory}

The weak asymptotic depends on the analytic properties of the measures considered. We need some basic results from potential theory which we summarize for the benefit of the reader.

\begin{lemma}
Let $E\subset\R$ be a regular compact set and $\phi$ a continuous function on $E$. Then, there exist a unique $\lambda\in\mathcal{M}_1(E)$ and a constant $w$ such that
\begin{equation*}
    V^\lambda(z) + \phi(z)
    \begin{cases}
    \leq w, & z\in \supp\lambda,\\
    \geq w, & z\in E.
    \end{cases}
\end{equation*}
\end{lemma}
In particular, equality holds on all $\supp\lambda$. In case that $E$ is not regular with respect to the Dirichlet problem, the second part of the statement is true except on a set $e$ such that $\capac(e)=0$. The proof of the lemma above in this context can be found in \cite[Theorem I.1.3]{safftotik}. When $E$ is regular, it is well known that this inequality except on a set of capacity zero implies the inequality for all points in the set, \cite[Theorem I.4.8]{safftotik}. The measure $\lambda$ is called the equilibrium measure in the presence of the external field $\phi$ on $E$, and $w$ the equilibrium constant.

We will need another tool for the proof of the asymptotic zero distribution of the polynomials $Q_{n,j}$. Different version of the next lemma can be found in \cite{gonrakh}, and \cite{Stahl_Totik}. The proof in \cite{gonrakh} was completed assuming that $\supp\sigma$ is an interval on which $\sigma'>0$ a.e. Theorem 3.3.3 in \cite{Stahl_Totik} does not cover the type of external field we need to consider. As stated here, the proof appears in \cite[Lemma 4.2]{lago_ulises_sorokin}

\begin{lemma}
\label{logarith_asymptotic}
Assume that $\sigma\in\Reg$ and $\supp\sigma\subset\R$ is regular. Let $\{\phi_n\}$, $n\in\Lambda\subset\Z_+$, be a sequence of positive continuous functions on $\supp\sigma$ such that
\begin{equation*}
    \lim_{n\in\Lambda}\frac{1}{2n}\log\frac{1}{|\phi_n(x)|} = \phi(x)>-\infty, \label{lim_log_phin}
\end{equation*}
uniformly on $\supp\sigma$. Let $\{q_n\}$, $n\in\Lambda$, be a sequence of monic polynomials such that $\deg q_n=n$ and
\begin{equation*}
    \int x^kq_n(x)\phi_n(x)\D\sigma(x) = 0, \qquad k=0,1,\ldots,n-1.
\end{equation*}
Then,
\begin{equation*}
    *\lim_{n\in\Lambda}\mu_{q_n}=\lambda
\end{equation*}
and
\begin{equation*}
    \lim_{n\in\Lambda}\left(\int |q_n(x)|^2\phi_n(x)\D\sigma(x)\right)^{1/2n} = e^{-w}
\end{equation*}
where $\lambda$ and $w$ are the equilibrium measure and equilibrium constant in the presence of the external field $\phi$ on $\supp\sigma$.
\end{lemma}

\section{Proof of the main results and some consequences}

\subsection{Proof of Theorem \ref{teo:1}}
For the proof of this result we make use of a technique introduced in \cite{GRS} for the study of the weak asymptotic of type II multiple orthogonal polynomials associated with generalized Nikishin systems (see also \cite{FL, Lago_Sergio_Ulises, lago_ulises_sorokin}).

\begin{proof}
The unit ball in the cone of positive Borel measures is weak star compact; therefore, it is sufficient to show that each  sequence of measures $\left(\mu_{Q_{n,j}}\right)_{n \geq N}, j=1,\ldots,m$, has only one accumulation point which coincides with the corresponding component of the vector equilibrium measure $\vec{\lambda}$ determined by the matrix $\mathcal{C}_\mathcal{N}$ on the system of compact sets $\supp \sigma_j$, $j=1,\ldots,m$.

Let $\Lambda$ be a sequence indices such that for each $j=1,\ldots,m$
\begin{equation*}
    *\lim_{n\in\Lambda}\mu_{Q_{n,j}} = \mu_j.
\end{equation*}
Notice that $\mu_j\in\mathcal{M}_1(E_j)$, $j=1,\ldots,m$. Taking into account that all the zeros of $Q_{n,j}$ lie in $\Delta_j$, it follows that
\begin{equation}
    \lim_{n\in\Lambda}|Q_{n,j}(z)|^{1/n} = \exp\left(-V^{\mu_j}(z)\right), \label{limit_nesimo_Qnj}
\end{equation}
uniformly on compact subsets of $\C\diff\Delta_j$.

The generating measures $\sigma_j, j=1,\ldots,m,$ have constant sign. Without loss of generality we can assume that they are positive.
Notice that  $\mathcal{A}_{n,m}=\pm Q_{n,m}T_n$, where $T_n\rightrightarrows T$ on compact subsets of $\mathbb{C}$ (recall that $a_{n,m}$ is monic). Hence, formula \ref{Orth_H_nj}, when $j=m-1$ becomes
\begin{equation*}
    \int x^\nu Q_{n,m}(x)\frac{T_n(x)\D \sigma_m(x)}{Q_{n,m-1}(x)},\qquad \nu=0,1,\ldots,n-D-1.
\end{equation*}
In order to use Lemma \ref{logarith_asymptotic}, write $\phi_n=T_n/Q_{n,m-1}$.  Then,
\begin{equation*}
    \lim_{n\in\Lambda}\frac{1}{2n}\log \phi_n(x) = \lim_{n\in\Lambda}\left(\frac{1}{2n}\log Q_{n,m-1} - \frac{1}{2n}\log T_n\right).
\end{equation*}
As $T_n\rightrightarrows T$ on $\supp\sigma_m$, where the polynomial $T$ has no zeros, we conclude that $0<b\leq|T_n|\leq B$, and $\frac{1}{2n}\log T_n \rightrightarrows 0$ uniformly on $\supp\sigma_m$. According to (\ref{limit_nesimo_Qnj}) we get
\begin{equation*}
    \lim_{n\in\Lambda}\frac{1}{2n}\log |Q_{n,m-1}(x)| = -\frac{1}{2}V^{\mu_{m-1}}(x),
\end{equation*}
uniformly on $\supp \sigma_m$. So,
\begin{equation*}
    \lim_{n\in\Lambda}\frac{1}{2n}\log \phi_n(x) = -\frac{1}{2}V^{\mu_{m-1}}(x)>-\infty.
\end{equation*}
Thus, from Lemma \ref{logarith_asymptotic} we deduce that $\mu_m$ is the unique solution of the extremal problem
\begin{equation}
    V^{\mu_m}(x) - \frac{1}{2}V^{\mu_{m-1}}(x)
    \begin{cases}
       = w_m, & x\in\supp(\mu_m),\\
    \geq w_m, & x\in \supp (\sigma_m),
    \end{cases} \label{extremal_Qnm}
\end{equation}
and
\begin{equation}
    \lim_{n\in\Lambda}\left(\int\frac{Q^2_{n,m}(x)}{|Q_{n,m-1}(x)|}\D\sigma_m(x)\right)^{1/2n} = e^{w_m}. \label{log_asymp_Qnm}
\end{equation}

Next, we prove by induction on decreasing values of $j$, that for all $j=1,\ldots,m$
\begin{equation}
    V^{\mu_j}(x)-\frac{1}{2}V^{\mu_{j-1}}(x)-\frac{1}{2}V^{\mu_{j+1}}(x)+w_{j+1}
    \begin{cases}
    =w_j, & x\in\supp \mu_j,\\
    \geq w_j, & x\in \supp \sigma_j,
    \end{cases}\label{extremal_Qnj}
\end{equation}
where $V^{\mu_0}\equiv V^{\mu_{m+1}}\equiv 0$, $w_{m+1}=0$, and
\begin{equation}
    \lim_{n\in\Lambda} \left(\int Q^2_{n,j}(x)\frac{|\mathcal{H}_{n,j}(x)|\D\sigma_j(x)}{|Q_{n,j-1}(x)Q_{n,m+1}(x)|}\right)^{1/2n}=e^{-w_j} \label{limit_norm2_Qnj}
\end{equation}
where $Q_{n,0}\equiv Q_{n,m+1}\equiv 1$.

Notice that for $j=m$ these relations are (\ref{log_asymp_Qnm}) and (\ref{extremal_Qnm}), and the initial step of the induction is settled. Suppose that the statement is true for $j+1\in\{3,\ldots,m\}$ and let us prove it for $j$. The step from $j=2$ to $j=1$ will be treated separately afterwards.

For $j=1,\ldots,m$ the orthogonality relations (\ref{Orth_H_nj}) can be expressed as
\begin{equation}
\label{Orth_H_nj_rewritten}
    \int x^\nu Q_{n,j}(x)\frac{\mathcal{H}_{n,j}(x)\D\sigma_{j}(x)}{Q_{n,j-1}(x)Q_{n,j+1}(x)}=0, \qquad \nu=0,1,\ldots,n-D-1,
\end{equation}
and using (\ref{H_mjIdentity_1}), $j=2,\ldots,m$
\begin{equation*}
    \int x^\nu Q_{n,j}(x)\left(\int \frac{Q^2_{n,j+1}(t)}{x-t}\frac{\mathcal{H}_{n,j+1}(t)\D\sigma_{j+1}(t)}{Q_{n,j}(t)Q_{n,j+2}(t)}\right)\frac{\D\sigma_{j}(x)}{Q_{n,j-1}(x)Q_{n,j+1}(x)} = 0,
\end{equation*}
for $\nu=0,1,\ldots,n-D-1$.

The limit in (\ref{limit_nesimo_Qnj}) gives us that
\begin{equation*}
    \lim_{n\in\Lambda}\frac{1}{2n}\log |Q_{n,j-1}(x)Q_{n,j+1}(x)| = -\frac{1}{2}V^{\mu_{j-1}}(x)-\frac{1}{2}V^{\mu_{j+1}}(x),\label{limit_log_Qnj-1Qnj+1}
\end{equation*}
uniformly on $\Delta_j$.

Set
\begin{equation*}
    K_{n,j+1}:=\left(\int Q^2_{n,j+1}(t)\frac{|\mathcal{H}_{n,j+1}(t)|\D\sigma_{j+1}(t)}{|Q_{n,j}(t)Q_{n,j+2}(t)|}\right)^{-1/2}.
\end{equation*}
It follows that for $x\in\Delta_j$
\begin{equation*}
    \frac{1}{\delta^*_{j+1}K^2_{n,j+1}} \leq \int \frac{Q^2_{n,j+1}(t)}{|x-t|} \frac{|\mathcal{H}_{n,j+1}(t)|\D\sigma_{j+1}(t)}{|Q_{n,j}(t)Q_{n,j+2}(t)|} \leq \frac{1}{\delta_{j+1}K^2_{n,j+1}}
\end{equation*}
where $0<\delta_{j+1}=\inf\{|x-t|:t\in\Delta_{j+1}, x\in\Delta_j\}\leq \max\{|x-t|:t\in\Delta_{j+1}, x\in\Delta_j\}=\delta^*_{j+1}<\infty$. Taking into consideration these inequalities, from the induction hypothesis, we obtain that
\begin{equation}
    \lim_{n\in\Lambda}\left(\int \frac{Q^2_{n,j+1}(t)}{|x-t|}\frac{|\mathcal{H}_{n,j+1}(t)|\D\sigma_{j+1}(t)}{|Q_{n,j}(t)Q_{n,j+2}(t)|}\right)^{1/2n} = e^{-w_{j+1}}. \label{limit_exp_wj+1}
\end{equation}
Taking (\ref{limit_log_Qnj-1Qnj+1}) and (\ref{limit_exp_wj+1}) into account, Lemma \ref{logarith_asymptotic} yields that $\mu_j$ is the unique solution of the extremal problem (\ref{extremal_Qnj}) and
\begin{equation*}
    \lim_{n\in\Lambda}\left(\iint \frac{Q^2_{n,j+1}(t)}{|x-t|}\frac{|\mathcal{H}_{n,j+1}(t)|\D\sigma_{j+1}(t)}{|Q_{n,j}(t)Q_{n,j+1}(t)|}\frac{Q^2_{n,j}(x)\D\sigma_j(x)}{|Q_{n,j-1}(x)Q_{n,j+1}(x)|}\right)^{1/2n} = e^{-w_j}.
\end{equation*}
As a consequence of (\ref{H_mjIdentity_1}), $j=2,\ldots,m-1$, the above formula reduces to (\ref{limit_norm2_Qnj}).

For $j=1$ formula (\ref{Orth_H_nj_rewritten}) becomes
\begin{equation*}
    \int x^\nu Q_{n,1}(x)\left(T(x)\int \frac{Q^2_{n,2}(t)}{x-t}\frac{\mathcal{H}_{n,j}(t)\D\sigma_2(t)}{Q_{n,1}(t)Q_{n,3}(t)}\right)\frac{\D\sigma_1(x)}{Q_{n,2}(x)} = 0,\; \nu = 0,\ldots,n-D-1.
\end{equation*}
From (\ref{limit_nesimo_Qnj}) we have  $\lim_{n\in\Lambda}\frac{1}{2n}\log|Q_{n,2}(x)| = -\frac{1}{2}V^{\mu_2}(x)$ uniformly on $\Delta_1$. Recall that $0<b\leq T(x)\leq B$ in $\Delta_1$, thus it follows that for $x\in\Delta_1$:
\begin{equation*}
    \frac{b}{\delta^*_2K^2_{n,2}}\leq |T(x)|\int \frac{Q^2_{n,2}(t)}{|x-t|}\frac{|\mathcal{H}_{n,j}(t)|\D\sigma_2(t)}{|Q_{n,1}(t)Q_{n,3}(t)|} \leq \frac{B}{\delta_2K^2_{n,2}}.
\end{equation*}
From here on, all the arguments used before work as well and the induction process is completed.

We can rewrite (\ref{extremal_Qnj}) as
\begin{equation}\label{vecequil2}
    V^{\mu_j}(x)-\frac{1}{2}V^{\mu_{j-1}}-\frac{1}{2}V^{\mu_{j+1}}(x)
    \begin{cases}
       = w_j', & x\in \supp\mu_j,\\
    \geq w_j', & x\in \supp \sigma_j,
    \end{cases}
\end{equation}
for $j=1,\ldots,m$, where
\begin{equation}
    w_j' = w_j-w_{j+1}, \qquad w_{m+1}=0. \label{relations_ws}
\end{equation}
(Recall that the terms with $V^{\mu_0}$ and $V^{\mu_{m+1}}$ do not appear when $j=0$ and $j=m$, respectively).
Now, \eqref{vecequil2} adopts the form of $\eqref{vecequil}$ which has only one solution. If follows that   $\vec{\lambda} = (\mu_1,\ldots,\mu_m)$ is the equilibrium solution for the vector potential problem determined by the interactions matrix $\mathcal{C}_{\mathcal{N}}$ on the system of compact sets $\supp \sigma_j$, $j=1,\ldots,m$ and $\omega^{\vec{\lambda}} = (w_1', \ldots, w_m')$ is the corresponding vector equilibrium constant. This is for any convergent subsequence; since the equilibrium problem does not depend on  $\Lambda$ and the solution is unique we obtain  (\ref{weak_Qnj}).

From the uniqueness of the vector equilibrium constant and (\ref{limit_norm2_Qnj}), we get
\begin{equation*}
    \lim_{n\to \infty}\left(\int Q^2_{n,j}(x) \frac{|\mathcal{H}_{n,j}(x)|\D\sigma_j(x)}{|Q_{n,j-1}(x)Q_{n,j+1}(x)|}\right)^{1/2n} = e^{-w_j},
\end{equation*}
On the other hand, from (\ref{relations_ws}) it follows that $w_m=\omega_m^{\vec{\lambda}}$ when $j=m$. Suppose that $w_{j+1}=\sum_{k=j+1}^m \omega^{\vec{\lambda}}_k$ where $j+1\in\{2,\ldots, m\}$. Then, according to (\ref{relations_ws})
\begin{equation*}
    w_j = w_j'+ w_{j+1} = \omega^{\vec{\lambda}}_j+w_{j+1} = \sum_{k=j}^m \omega^{\vec{\lambda}}_k
\end{equation*}
and \eqref{limit_nesimo}  immediately follows.
\end{proof}

\subsection{Proof of Theorem \ref{logarithmic_asymptotic_Anj}}

\begin{proof}
Since $\mathcal{A}_{n,m} = \pm Q_{n,m}T_n$ and $T_n\rightrightarrows T$ on compact subsets of $\C$,  (\ref{weak_Qnj}) implies
\begin{equation*}
    \lim_{n\to \infty}|\mathcal{A}_{n,m}(z)|^{1/n} = \exp\left(-V^{\lambda_m}\right), \qquad z\in \mathcal{K}\subset\C\diff(\Delta_m\cup Z).
\end{equation*}
For $j=1,\ldots,m-1$, from (\ref{H_mjIdentity_1}) we have
\begin{equation}
    \mathcal{A}_{n,j}(z) = \frac{Q_{n,j}(z)}{Q_{n,j+1}(z)}\int \frac{Q^2_{n,j+1}(x)}{z-x}\frac{\mathcal{H}_{n,j+1}(x)\D\sigma_{j+1}(x)}{Q_{n,j}(x)Q_{n,j+2}(x)}, \label{Anj_Integral}
\end{equation}
where $Q_{n,0}\equiv Q_{n,m+1}\equiv 1$. Now, (\ref{weak_Qnj}) implies
\begin{equation*}
    \lim_{n\to \infty} \left|\frac{Q_{n,j}(z)}{Q_{n,j+1}(z)}\right|^{1/n} = \exp\left(V^{\lambda_{j+1}}(z)-V^{\lambda_j}(z)\right),\qquad \mathcal{K}\subset\C\diff(\Delta_j\cup \Delta_{j+1})
\end{equation*}
(we also use that the zeros of $Q_{n,j}$ and $Q_{n,j+1}$ lie in $\Delta_j$ and $\Delta_{j+1}$, respectively). It remains to find the $n$-th root asymptotic behavior of the integral.

Fix a compact set $\mathcal{K}\subset\C\diff\Delta_{j+1}$. It is easy to verify that
\begin{equation}
    \frac{C_1}{K_{n,j+1}^2}\leq \left|\int \frac{Q^2_{n,j+1}(x)}{z-x}\frac{\mathcal{H}_{n,j+1}(x)\D\sigma_{j+1}(x)}{Q_{n,j}(x)Q_{n,j+1}(x)}\right| \leq \frac{C_2}{K_{n,j+1}^2}, \label{bounds_integral}
\end{equation}
where
\begin{equation*}
    C_1 = \frac{\min\{\max\{|u-x|, |v|: z=u+iv\}: z\in\mathcal{K}, x\in \Delta_{j+1}\}}{\max\{|z-x|^2: z\in\mathcal{K}, x\in\Delta_{j+1}\}}
\end{equation*}
and
\begin{equation*}
    C_2 = \frac{1}{\min\{|z-x|:z\in\mathcal{K}, x\in\Delta_{j+1}\}}<\infty.
\end{equation*}
Taking into account (\ref{limit_nesimo}) we get
\begin{equation}
    \lim_{n\to \infty} \left|\int \frac{Q^2_{n,j+1}(x)}{z-x}\frac{\mathcal{H}_{n,j+1}(x)\D\sigma_{j+1}(x)}{Q_{n,j}(x)Q_{n,j+2}(x)}\right|^{1/n} = \exp\left(-2\sum_{k=j+1}^m \omega^{\vec{\lambda}}_k\right). \label{limit_integral}
\end{equation}
From (\ref{Anj_Integral})-(\ref{limit_integral}) we deduce (\ref{limit_nrooth_Anj}). Finally, notice that
\begin{equation*}
     \lim_{n\to \infty} |T(z)\mathcal{A}_{n,0}(z)|^{1/n} = \lim_{n\to \infty} |\mathcal{A}_{n,0}(z)|^{1/n}
\end{equation*}
for all $z\in\C\diff(\Delta_1\cup Z)$, in case that the first limit exists. This last statement holds, and its proof follows easily using the same arguments as above.
\end{proof}

\subsection{Consequences}

Let us find the logaritmic asymptotic of the polynomials $a_{n,j}, j=0,\ldots,m$.

\begin{corollary} Under the assumnptions of Theorem \ref{teo:1},
\begin{equation}
    \lim_{n\to \infty} |a_{n,j}(z)|^{1/n} = A_m(z), \qquad j=1,\ldots,m,\label{limit_nrooth_anj}
\end{equation}
uniformly on compact subsets of $\C\diff(\Delta_m\cup Z)$.
\end{corollary}
\begin{proof}
In the proof of Theorem \ref{logarithmic_asymptotic_Anj}  we  obtained (\ref{limit_nrooth_anj}) for $j=m$. Now, recall that the function $\widehat{s}_{m,j+1}$ never equals zero in $\C\diff(\Delta_m\cup Z)$; therefore, for the remaining values of $j$, the limit (\ref{limit_nrooth_anj}) is an immediate consequence of (\ref{limit_nrooth_anj}) for $j=m$ and (\ref{convergencia}).
\end{proof}

Regarding \eqref{limit_nrooth_anj} for $j=0$, aside from $Z$ we would have to exclude  from $\C\diff \Delta_m$ all the points where $f=0$.

Our next goal is to  produce estimates of the rate of convergence in (\ref{convergencia}). First we prove

\begin{corollary}
\label{corollary_logarithmic_asym}
Under the assumptions of Theorem \ref{teo:1}, for $k=1,\ldots,m$ and $j=0,\ldots,k-1$, we have
\begin{multline}
    \limsup_{n\to \infty} \left|\frac{\mathcal{A}_{n,j}(z)}{\mathcal{A}_{n,k}(z)}\right|^{1/n} \leq \\ \exp\left(-V^{\lambda_{k+1}}(z)+V^{\lambda_{k}}(z)+V^{\lambda_{j+1}}(z)-V^{\lambda_{j}}(z)-2\sum_{\ell=j+1}^k \omega_\ell^{\vec{\lambda}}\right), \label{inequality_nrooth_ratio_Anj}
\end{multline}
uniformly on compact subsets $\mathcal{K}\subset\C\diff(\Delta_k\cup\Delta_{j+1})$, and
\begin{multline}
    \lim_{n\to \infty} \left|\frac{\mathcal{A}_{n,j}(z)}{\mathcal{A}_{n,k}(z)}\right|^{1/n} =\\
    \exp\left(-V^{\lambda_{k+1}}(z)+V^{\lambda_{k}}(z)+V^{\lambda_{j+1}}(z)-V^{\lambda_{j}}(z)-2\sum_{\ell=j+1}^k \omega_\ell^{\vec{\lambda}}\right), \label{limit_nrooth_ratio_Anj}
\end{multline}
uniformly on compact subsets of $\mathcal{K}\subset\C\diff(\Delta_j\cup\Delta_{j+1}\cup\Delta_k\cup\Delta_{k+1})$. If $j=0$ or $k=m$ we must also delete from $\C$ the zeros of $T$ in order that \eqref{inequality_nrooth_ratio_Anj} and \eqref{limit_nrooth_ratio_Anj} remain valid. For $k=1,\ldots,m$
\begin{equation}
    -V^{\lambda_{k+1}}(z)+2V^{\lambda_{k}}(z)-V^{\lambda_{k-1}}(z)-2 \omega_k^{\vec{\lambda}}<0, \qquad z\in\C\diff\Delta_k \label{potentials_inequality_0}
\end{equation}
(by convention, $V^{\lambda_0}\equiv V^{\lambda_{m+1}}\equiv 0$). If $k>j+1$
\begin{equation}
    -V^{\lambda_{k+1}}(z)+V^{\lambda_{k}}(z)+V^{\lambda_{j+1}}(z)-V^{\lambda_{j}}(z)-2\sum_{\ell=j+1}^k \omega_\ell^{\vec{\lambda}}<0, \qquad z\in \C,\label{inequality_sumpotentials}
\end{equation}
which implies that the sequence $\{\mathcal{A}_{n,j}/\mathcal{A}_{n,k}\}$ converges to zero with geometric rate on each compact subset of $\C\diff(\Delta_k\cup\Delta_{j+1})$ ($\C\diff(\Delta_k\cup\Delta_{j+1} \cup Z)$ if $k=m$ or $j=0$).
\end{corollary}
\begin{proof}
Fix $k\in\{1,\ldots,m-1\}$ and $j\in\{1,\ldots,k-1\}$.  Using (\ref{Anj_Integral}) we get
\begin{equation}
    \frac{\mathcal{A}_{n,j}(z)}{\mathcal{A}_{n,k}(z)} = \frac{Q_{n,j}(z)Q_{n,k+1}(z)}{Q_{n,j+1}(z)Q_{n,k}(z)} \frac{\int \frac{Q^2_{m,j+1}(z)}{z-x}\frac{\mathcal{H}_{n,j}(x)\D\sigma_{j+1}(z)}{Q_{n,j}(z)Q_{n,j+2}(z)}}{\int \frac{Q^2_{m,k+1}(z)}{z-x}\frac{\mathcal{H}_{n,k}(x)\D\sigma_{k+1}(z)}{Q_{n,k}(z)Q_{n,k+2}(z)}} \label{ratio_Anj_integralrep*}
\end{equation}
From (\ref{weak_Qnj}) it follows that uniformly on each compact subset $\mathcal{K}\subset\C\diff(\Delta_j\cup\Delta_{j+1}\cup\Delta_k\cup\Delta_{k+1})$ we have
\begin{equation*}
    \lim_{n\to \infty} \left|\frac{Q_{n,j}(z)Q_{n,k+1}(z)}{Q_{n,j+1}(z)Q_{n,k}(z)}\right|^{1/n} = \exp\left(-V^{\lambda_{k+1}}(z)+V^{\lambda_k}(z)+V^{\lambda_{j+1}}(z)-V^{\lambda_j}\right),
\end{equation*}
and taking into account (\ref{limit_integral}), from (\ref{ratio_Anj_integralrep*}) we deduce (\ref{limit_nrooth_ratio_Anj}).

Now, from the principle of descent (see \cite[Appendix III]{Stahl_Totik}), locally uniformly on $\C$ we have
\begin{equation*}
    \limsup_{n\to \infty} |Q_{n,j}(z)Q_{n,k+1}|^{1/n} \leq \exp\left(-V^{\lambda_{k+1}}(z)-V^{\lambda_{j}}(z)\right).
\end{equation*}
Using the lower bound in (\ref{bounds_integral}) (with $j$ replaced by $k$) to estimate the integral in the denominator of (\ref{ratio_Anj_integralrep*}) from below and the previous remarks, (\ref{inequality_nrooth_ratio_Anj}) readily follows.

If $k=m$ and $j=1,\ldots,m-1$ in place of \eqref{ratio_Anj_integralrep*} we use the representation
\begin{equation*}
    \left|\frac{\mathcal{A}_{n,j}(z)}{\mathcal{A}_{n,m}(z)}\right| = \left|\frac{Q_{n,j}(z) }{Q_{n,j+1}(z)Q_{n,m}(z)T_n(z)} {\int \frac{Q^2_{m,j+1}(z)}{z-x}\frac{\mathcal{H}_{n,j}(x)\D\sigma_{j+1}(z)}{Q_{n,j}(z)Q_{n,j+2}(z)}}\right|,
\end{equation*}
where $T_n \rightrightarrows T$ and then argue as above. If $j=0$ and $k=1,\ldots,m$ the treatment is similar and is left to the reader.

According to (\ref{extremal_Qnj}), for $k=1,\ldots,m$ we have
\begin{equation}
    -V^{\lambda_{k+1}}(z)+2V^{\lambda_{k}}(z)-V^{\lambda_{k-1}}(z)-2 \omega_k^{\vec{\lambda}}=0,\qquad z\in\supp\lambda_k. \label{potentials_equality_0}
\end{equation}
Recall that all the measures $\lambda_k$ are probability, hence for each $k=2,\ldots,m-1$ the function $ -V^{\lambda_{k+1}}(z)+2V^{\lambda_{k}}(z)-V^{\lambda_{k-1}}(z)-2 \omega_k^{\vec{\lambda}}$ is harmonic at $z=\infty$, and is subharmonic in $\C\diff\supp\lambda_k$. Using maximum principle for subharmonic functions we obtain (\ref{potentials_inequality_0}).

When $k=1$, the left hand of (\ref{potentials_equality_0}) becomes $-V^{\lambda_2}(z)+2V^{\lambda_1}-2\omega_1^{\vec{\lambda}}$ which is subharmonic in $\C\diff\supp\lambda_1$ and also subharmonic at $\infty$ since
\begin{equation*}
    \lim_{n\to \infty} \left(-V^{\lambda_2}(z)+2V^{\lambda_1}-2\omega_1^{\vec{\lambda}}\right) = -\infty.
\end{equation*}
Therefore, we can also use the maximum principle to derive (\ref{potentials_inequality_0}). The case $k=m$ is completely analogous to the case $k=1$.

When $k>j+1$ we can write
\begin{multline*}
    -V^{\lambda_{k+1}}(z) + V^{\lambda_k}(z) + V^{\lambda_{j+1}}(z) - V^{\lambda_j}(z) - 2\sum_{\ell=j+1}^m\omega_k^{\vec{\lambda}} =\\
    \sum_{\ell=j+1}^k \left(-V^{\lambda_{\ell+1}}(z)+2V^{\lambda_\ell}(z)-V^{\lambda_{\ell-1}}(z)-2\omega_\ell^{\vec{\lambda}}\right),
\end{multline*}
and this sum contains at least two terms because $k>j+1$. Each term is less than or equal to zero in all $\C$ and so too the whole sum. To prove that it is strictly negative it is sufficient to show that at each point there is at least one negative term in the sum. Let us assume that there is a $z_0\in \C$ such that
\begin{equation*}
    -V^{\lambda_{\ell+1}}(z_0)+2V^{\lambda_\ell}(z_0)-V^{\lambda_{\ell-1}}(z_0)-2\omega_\ell^{\vec{\lambda}}=0, \qquad \ell=j+1,\ldots,k.
\end{equation*}
By what was proved above, this implies that $z_0\in\cap_{\ell=j+1}^k\Delta_\ell$. However, this is impossible because consecutive intervals in a Nikishin system are disjoint. From (\ref{inequality_nrooth_ratio_Anj}) and (\ref{inequality_sumpotentials}) the final statement is deduced.
\end{proof}

Using Corollary \ref{corollary_logarithmic_asym} we can recover the functions $\widehat{s}_{m-1,j+1}$, $j=1,\ldots,m-2$.
\begin{corollary}
\label{corallary_logarithmic_asym1}
Under the assumptions of Theorem \ref{logarithmic_asymptotic_Anj}, for each $j=1,\ldots,m-2$ we have
\begin{equation}
    \lim_{n\to \infty} \frac{(a_{n,j}-a_{n,m}\widehat{s}_{m,j+1})(z)}{(a_{n,m-1}-a_{n,m}\widehat{s}_{m,m})(z)} = \widehat{s}_{m-1,j+1}(z),\label{recover_measures}
\end{equation}
and
\begin{equation}
    \lim_n \frac{\left(a_{n,0}+\sum_{k=1}^{m}(-1)^k a_{n,k}r_k\right)-a_{n,m}\hat{s}_{m,1}}{(a_{n,m-1}-a_{n,m}\widehat{s}_{m,m})(z)}=\hat{s}_{m-1,1}.\label{recover_measures_0}
\end{equation}
uniformly on each compact subset of $\C \setminus \cup_{\ell=j+1}^m\Delta_\ell$.
\end{corollary}
\begin{proof}

Direct computation or \cite[Lemma 2.1]{Lago_Sergio_Jacek} allows to deduce the formula
\begin{equation}
\label{identity_for_recover}
    \mathcal{A}_{n,j} + \sum_{k=j+1}^{m-1}(-1)^{k-j}\widehat{s}_{k,j+1}\mathcal{A}_{n,k} = (-1)^j(a_{n,j}-a_{n,m}\widehat{s}_{m,j+1}).
\end{equation}
The formula holds at all points where both sides are meaningful. Dividing by $\mathcal{A}_{n,m-1}$ we get
\begin{multline*}
    \frac{\mathcal{A}_{n,j}}{\mathcal{A}_{n,m-1}} + \sum_{k=j+1}^{m-2}(-1)^{k-j}\widehat{s}_{k,j+1}\frac{\mathcal{A}_{n,k}}{\mathcal{A}_{n,m-1}} + (-1)^{m-1-j}\widehat{s}_{m-1,j+1} =\\
    (-1)^{m-1+j}\frac{(a_{n,j}-a_{n,m}\widehat{s}_{m,j+1})(z)}{(a_{n,m-1}-a_{n,m}\widehat{s}_{m,m})(z)}.
\end{multline*}
In order to obtain (\ref{recover_measures}), it remains to take limit on both sides and make use of the fact that the ratios $\mathcal{A}_{n,k}/\mathcal{A}_{n,m-1}$ uniformly tend to zero on compact subsets of $\C\diff\cup_{\ell=j+1}^m\Delta_\ell$.

To prove (\ref{recover_measures_0}) instead of \eqref{identity_for_recover} we use the formula
\begin{equation}
\label{nth_root_An0}
    \mathcal{A}_{n,0} + \sum_{k=1}^{m-1} (-1)^k \hat{s}_{k,1}\mathcal{A}_{n,k} = \left(a_{n,0}+\sum_{k=1}^{m}(-1)^k a_{n,k}r_k\right)-a_{n,m}\hat{s}_{m,1},
\end{equation}
which is obtained similarly. Dividing by $\mathcal{A}_{n,m-1}$ and taking limit we complete the proof.
\end{proof}

We wish to mention that the convergence in (\ref{recover_measures})-\eqref{recover_measures_0} occurs with geometric rate, as a result of (\ref{inequality_nrooth_ratio_Anj}) and (\ref{inequality_sumpotentials}).

Using Corollary \ref{corallary_logarithmic_asym1} we can give explicit expressions for the exact rate of convergence of the limits (\ref{convergencia}).

\begin{theorem}
\label{convergence_speed}
Under the assumptions of Theorem \ref{teo:1}, for each $j=1,\ldots,m-1$:
\begin{equation}
    \label{geometric_speed}
    \lim_{n\to \infty} \left| \frac{a_{n,j}(z)}{a_{n,m}(z)}-\widehat{s}_{m,j+1}(z)  \right|^{1/n} = \exp\left(2V^{\lambda_m}(z)-V^{\lambda_{m-1}}(z)-2\omega_m^{\vec{\lambda}}\right)
\end{equation}
and
\begin{equation}
    \label{geometric_speed_an0}
    \limsup_{n\to \infty} \left| \frac{a_{n,0}(z)}{a_{n,m}(z)}-f(z) \right|^{1/n} \leq \exp\left(2V^{\lambda_m}(z)-V^{\lambda_{m-1}}(z)-2\omega_m^{\vec{\lambda}}\right)
\end{equation}
uniformly on each compact subset $\mathcal{K}\subset  \mathbb{C}\setminus  (\cup_{\ell=j+1}^{m}\Delta_{\ell}\cup Z)$.
\end{theorem}
\begin{proof}
Our starting point is (\ref{identity_for_recover}), but now we divide it by $\mathcal{A}_{n,m} = (-1)^ma_{n,m}$. We get
\begin{equation*}
    \left|\frac{\mathcal{A}_{n,j}}{\mathcal{A}_{n,m}} + \sum_{k=j+1}^{m-1}(-1)^{k-j}\widehat{s}_{k,j+1}\frac{\mathcal{A}_{n,k}}{\mathcal{A}_{n,m}}\right| = \left|\frac{a_{n,j}}{a_{n,m}}-\widehat{s}_{m,j+1}\right|,
\end{equation*}
which is equivalent to
\begin{equation*}
\label{first_speed_estimate}
    \left|\frac{\mathcal{A}_{n,m-1}}{\mathcal{A}_{n,m}}\right|\left|\frac{\mathcal{A}_{n,j}}{\mathcal{A}_{n,m-1}} + \sum_{k=j+1}^{m-1}(-1)^{k-j}\widehat{s}_{k,j+1}\frac{\mathcal{A}_{n,k}}{\mathcal{A}_{n,m-1}}\right| = \left|\frac{a_{n,j}}{a_{n,m}}-\widehat{s}_{m,j+1}\right|.
\end{equation*}
Now, $\widehat{s}_{m-1,j+1}(z)\neq 0$, $z\in\C\diff\Delta_{m-1}$; consequently, $\lim_{n\to \infty} |\widehat{s}_{m-1,j+1}(z)|^{1/n}=1$ uniformly on compact subsets of $\C\diff\Delta_{m-1}$. Therefore,
\[
\lim_{n\to \infty} \left|\frac{(-1)^j \mathcal{A}_{n,j}}{\mathcal{A}_{n,m-1}} + \sum_{k=j+1}^{m-1} (-1)^{k} \widehat{s}_{k,j+1} \frac{\mathcal{A}_{n,k}}{\mathcal{A}_{n,m-1}} \right|^{1/n}= 1,
\]
uniformly on compact subsets of   $\mathbb{C}\setminus  \cup_{\ell=j+1}^{m}\Delta_{\ell}.$ On the other hand, from \eqref{limit_nrooth_ratio_Anj}
\[ \lim_{n\to \infty} \left|\frac{\mathcal{A}_{n,m-1}}{\mathcal{A}_{n,m}}\right|^{1/n} =   \exp\left(2V^{\lambda_m}(z)-V^{\lambda_{m-1}}(z)-2\omega_m^{\vec{\lambda}}\right),
\]
uniformly on compact subsets of $\C \setminus (\Delta_{m-1} \cup \Delta_m \cup Z)$. These relations together imply \eqref{geometric_speed}.

To estimate the speed of convergence of the quotients $a_{n,0}/a_{n,m}$ we use equality (\ref{nth_root_An0}). Dividing it by $\mathcal{A}_{n,m}$   we get
\begin{equation*}
    \left|\frac{\mathcal{A}_{n,m-1}}{\mathcal{A}_{n,m}}\right|\left|\frac{\mathcal{A}_{n,0}}{\mathcal{A}_{n,m-1}} + \sum_{k=1}^{m-1} (-1)^k\hat{s}_{k,1}\frac{\mathcal{A}_{n,k}}{\mathcal{A}_{n,m-1}}\right| = \left|\frac{a_{n,0}}{a_{n,m}} + \sum_{k=1}^{m}(-1)^k\frac{a_{n,k}}{a_{n,m}}r_k -\hat{s}_{m,1}\right|.
\end{equation*}
Arguing as above this equality implies that
\begin{equation} \label{geometric_speed-1}
    \lim_{n\to \infty} \left|\frac{a_{n,0}(z)}{a_{n,m}(z)} + \sum_{k=1}^{m}(-1)^k\frac{a_{n,k}(z)}{a_{n,m}(z)}r_k(z) -\hat{s}_{m,1}(z) \right|^{1/n} = \exp\left(2V^{\lambda_m}(z)-V^{\lambda_{m-1}}(z)-2\omega_m^{\vec{\lambda}}\right),
\end{equation}
uniformly on each compact subset $\mathcal{K}\subset  \mathbb{C}\setminus  (\cup_{\ell=1}^{m}\Delta_{\ell}\cup Z)$.

On the other hand, using the formula for $f$, we obtain
\begin{equation*}
    \left|\frac{a_{n,0}}{a_{n,m}} + \sum_{k=1}^{m}(-1)^k\frac{a_{n,k}}{a_{n,m}}r_k -\hat{s}_{m,1}\right| = \left|\left(\frac{a_{n,0}}{a_{n,m}}-f\right) + \sum_{k=1}^{m-1}(-1)^k\left(\frac{a_{n,k}}{a_{n,m}}-\hat{s}_{m,k+1}\right)r_k\right| \geq
    \]
    \[
    \left|\frac{a_{n,0}}{a_{n,m}}-f\right| - \sum_{k=1}^{m-1}\left|\left(\frac{a_{n,k}}{a_{n,m}}-\hat{s}_{m,k+1}\right)r_k\right|;
\end{equation*}
that is,
\begin{multline*} \left|\frac{a_{n,0}(z)}{a_{n,m}(z)}-f(z)\right| \leq \\
 \left|\frac{a_{n,0}(z)}{a_{n,m}(z)} + \sum_{k=1}^{m}(-1)^k\frac{a_{n,k}(z)}{a_{n,m}(z)}r_k (z) -\hat{s}_{m,1}(z)\right| +
\sum_{k=1}^{m-1}\left|\left(\frac{a_{n,k}(z)}{a_{n,m}(z)}-\hat{s}_{m,k+1}(z)\right)r_k(z)\right|
\end{multline*}
This inequality, together with  \eqref{geometric_speed} and \eqref{geometric_speed-1}, implies \eqref{geometric_speed_an0}.
\end{proof}


\begin{thebibliography}{99}

\bibitem{3} R. Beals, D.H. Sattinger, and J. Szmigielski. Multi-peakons and the classical moment problem. Adv. in Math. {\bf 154} (2000), 229-257.


\bibitem{bello_lago} M. Bello Hern\'andez, G. L\'opez Lagomasino, and J. Mínguez Ceniceros. Fourier-Pad\'e approximants of Angelesco systems. Constr. Approx. \textbf{26} (2007), 339-359.






\bibitem{FL} U. Fidalgo Prieto, G. L\'opez Lagomasino. Rate of convergence of generalized Hermite-Padé approximants of Nikishin systems. Constr. Approx., 23 (2006), 165-196.

\bibitem{Lago_Sergio_Ulises} U. Fidalgo Prieto, G. L\'opez Lagomasino, and S. Medina Peralta. Asymptotic of Cauchy biorthogonal polynomials. Mediterr. J. Math. (2020) 17: 22.

\bibitem{lago_ulises_sorokin} U. Fidalgo Prieto, G. L\'opez Lagomasino, and V. N. Sorokin. Mixed type multiple orthogonal polynomials for two Nikishin systems.  Constr. Approx. \textbf{32} (2010), 255-306.



\bibitem{gonrakh} A. A. Gonchar and E. A. Rakhmanov. On convergence of simultaneous Pad\'e approximants for systems of functions of Markov type.
 Proc. Steklov Inst. Math. \textbf{157} (1983), 31-50.


\bibitem{GRS} A.A. Gonchar, E.A. Rakhmanov, and V.N. Sorokin.   Hermite--Pad\'e approximants for systems of Markov--type functions.  Sb. Math. {\bf 188} (1997), 33-58.


\bibitem{GLM} L. G. Gonz\'alez Ricardo, G. L\'opez Lagomasino, and S. Medina Peralta.
 On the convergence of multi-level Hermite-Pad\'e approximants. arxiv  2001.08276






\bibitem{Lago_Sergio1} G. L\'opez Lagomasino and S. Medina Peralta. On the convergence of type \textsc{i} Hermite-Pad\'e approximants for a class of meromorphic functions.  J. Comput. Appl. Math. \textbf{284} (2015), 216-227.

\bibitem{Lago_Sergio_Jacek} G. L\'opez Lagomasino, S. Medina Peralta, and J. Szmigielski. Mixed type Hermite-Pad\'e approximation inspired by the Degasperis-Procesi equation.  Adv. Math. \textbf{349} (2019), 813–838.

\bibitem{LS} H. Lundmark, J. Szmigielski. Degasperis-Procesi peakons and the discrete cubic string. Int. Math.
Res. Pap. 2 (2005).


\bibitem{nikishin} E. M. Nikishin. On simultaneous Pad\'e approximants.  Math. USSR Sb. \textbf{41} (1982), 409-425.



\bibitem{safftotik} E. B. Saff and V. Totik.  Logarithmic Potentials with External Fields. Comprehensive
Studies in Mathematics 316. Springer, New York, USA, 1997.

\bibitem{Stahl_Totik} H. Stahl and V. Totik.  General Orthogonal Polynomials. Encyclopedia of Mathematics and its applications 43. Cambridge University Press, Cambridge, UK, 1992.


\end{thebibliography}
\end{document}